\documentclass[12pt]{amsart}
\usepackage{fancyhdr}
\usepackage{stmaryrd}
\usepackage{calrsfs}
\DeclareMathAlphabet{\pazocal}{OMS}{zplm}{m}{n}

\usepackage{amsmath}
\usepackage[nospace,noadjust]{cite}
\usepackage{amsfonts}
\usepackage{amssymb,enumerate}
\usepackage{amsthm}
\usepackage{cite}
\usepackage{comment}
\usepackage{color}
\usepackage[all]{xy}
\usepackage{hyperref}
\usepackage{lineno}
\usepackage{stmaryrd}
\usepackage{varwidth}

\usepackage{mathtools}
\usepackage{bm}
\usepackage{graphicx}
\usepackage{psfrag}
\usepackage{url}

\usepackage{fancyhdr}
\usepackage{tikz-cd}

\usepackage{dirtytalk}
\usepackage{float}
\usepackage{mathrsfs}

\usepackage{todonotes}
\usepackage[capitalize,nameinlink]{cleveref}

\newtheorem*{maintheorem*}{Main Theorem}
\newtheorem{theorem}{Theorem}[section]
\newtheorem{prop}[theorem]{Proposition}
\newtheorem{question}[theorem]{Question}
\newtheorem{conj}[theorem]{Conjecture}

\theoremstyle{definition}

\numberwithin{equation}{section}

\newcommand{\EE}{\mathbb{E}}
\newcommand{\PP}{\mathbb{P}}

\hypersetup{
	pdftitle={FD},
	colorlinks=true,
	linkcolor=blue,
	citecolor=cyan,
	urlcolor=wine
}

\begin{document}
	
	\mbox{}
	\title{Monochromatic components with many edges in random graphs}

	\author{Hannah Fox, Sammy Luo$^*$}
    \thanks{$^*$Massachusetts Institute of Technology. Research supported by NSF Award No. 2303290. sammyluo@mit.edu}

	\begin{abstract}
		In an $r$-coloring of edges of the complete graph on $n$ vertices, how many edges are there in the largest monochromatic connected component? A construction of Gyárfás shows that for infinitely many values of $r$, there exist colorings where all monochromatic components have at most $\left(\frac{1}{r^2-r}+o(1)\right)\binom{n}{2}$ edges. Conlon, Luo, and Tyomkyn conjectured that components with at least this many edges are attainable for all $r \ge 3$. This was proven by Luo for $r=3$, along with a lower bound of $\frac{1}{r^2-r+\frac54}{n\choose 2}$
		for all $r\ge 2$, and by Conlon, Luo, and Tyomkyn for $r=4$.
  
		In this paper, we look at extensions of this problem where the graph being $r$-colored is a sparse random graph or a graph of high minimum degree. By extending several intermediate technical results from previous work in the complete graph setting, we prove analogues of the bound for general $r$ in both the sparse random setting and the high minimum degree setting, as well as the bound for $r=3$ in the latter setting.
	\end{abstract}
	\medskip

	\maketitle


\bigskip

\section{Introduction}
\label{sec:intro}


A classical observation of Erdős and Rado is that in any 2-coloring of the edges of the complete graph $K_n$, one of the two color classes consists of a single connected component on all $n$ vertices. 
Gy\'arf\'as \cite{gya} extended this result, proving that for $r\ge 2$, in any $r$-coloring of the edges of $K_n$, there is 
a monochromatic connected component with at least $\frac{n}{r-1}$ vertices. He also exhibited a construction showing that this bound is tight for 
arbitrarily large $n$ and infinitely many values of $r$. 

Since then, Gy\'arf\'as's result has been further generalized in many different ways. One notable example is due to Bal and DeBiasio \cite{bd}, who showed that a similar result holds when $K_n$ is replaced by an Erd\H{o}s-Rényi random graph $G(n,p)$, which is a random graph on $n$ vertices where each pair of vertices is adjacent with an independent probability of $p$.

\begin{theorem}[\cite{bd}]\label{thm:bd}
    For all $r\ge 2$ and sufficiently small $\varepsilon >0$, there exists $C$ such that for $p\ge \frac{C}{n}$, with high probability every $r$-coloring of $G(n,p)$ contains a monochromatic connected component with at least $(1-\varepsilon)\frac{n}{r-1}$ vertices.
\end{theorem}

We can also ask analogous questions about the largest number of \emph{edges} in a monochromatic component.

\begin{question}\label{q:rcolor}
    If the edges of $K_n$ are colored with $r$ colors, what is the largest number of edges that some monochromatic component is guaranteed to have?
\end{question}
This question was first studied by Conlon and Tyomkyn \cite{ct}, who showed that in the case $r=2$, every such coloring contains a monochromatic component with at least $\frac{2}{9}n^2+O(n^{3/2})$ edges. Later work by Luo \cite{sluo} and by Conlon, Luo, and Tyomkyn \cite{clt} resolved the cases $r=3$ and $r=4$ respectively, showing that when the edges of $K_n$ are colored with $3$ (resp. $4$) colors, there always exists a monochromatic connected component with at least $\frac{1}{6}\binom{n}{2}$ (resp. $\frac{1}{12}\binom{n}{2}$) edges, and that these bounds are tight.

The lower bounds in the $r=3$ and $r=4$ cases match the upper bounds arising from a close examination of Gy\'arf\'as's construction, which shows that for infinitely many values of $r$ (in particular, whenever $r-1$ is a prime power), there is an $r$-coloring of $K_n$ where every monochromatic component has at most $(\frac{1}{r(r-1)}+o(1))\binom{n}{2}$ edges. Conlon, Luo, and Tyomkyn conjectured that a matching lower bound holds for every value of $r$.

\begin{conj}[\cite{clt}]\label{conj:rKn}
    For any positive integers $n,r$ with $r\ge 3$, in every $r$-coloring of the edges of $K_n$, there is some monochromatic connected component with at least $\frac{1}{r(r-1)}\binom{n}{2}$ edges.  
\end{conj}
In \cite{sluo}, Luo proves a general lower bound that approximates the conjectured lower bound within a factor of $1-O(\frac{1}{r^2})$.

\begin{theorem}[\cite{sluo}]\label{thm:rKn}
In any $r$-coloring of the edges of $K_n$ for $r\ge 2$, there is a monochromatic connected component with at least $\frac{1}{r^2-r+\frac54}{n\choose 2}$ edges.
\end{theorem}

In analogy to the generalization of Gy\'arf\'as's original result to \cref{thm:bd}, in this paper we study a generalization of \cref{q:rcolor} for random graphs: In an $r$-coloring of the edges of the random graph $G(n, p)$, what lower bounds 
can we obtain (with high probability) for the number of edges in the largest monochromatic component?

The answer to this question naturally depends on our choice of the density $p$. We may ask if there is a \textit{threshold} for values of $p$ above which the answer to the above question about $G(n,p)$ starts to behave in a certain way with high probability, e.g. when our bounds on the number of edges in the largest monochromatic component start to behave similarly to the analogous bounds in the complete graph setting.

Our main result is an analogue of \cref{thm:rKn} in the setting of $G(n, p)$ for any $p= \omega(\frac{\log n}{n})$.
\begin{theorem} 
    \label{thm:main}
    For any fixed integer $r\ge 2$ and any $p=\omega(\frac{\log n}n)$, 
    $G:=G(n, p)$ satisfies the following property with high probability: In every $r$-coloring of the edges of $G$,    
    there is a monochromatic connected 
    component with at least $\frac{1}{r^2-r+\frac54}|E(G)|$ edges.
\end{theorem}
We prove this result by extending the argument used in proving \cref{thm:rKn}, including several technical lemmas, to the sparse random graph setting. Applying similar extensions yields a sparse random analogue of the lower bound for $r=3$ as well.

The paper proceeds as follows. We begin in \cref{sec:prelims} by providing some background on previous work on \cref{q:rcolor} for complete graphs, including some technical tools that we will adapt for the random graph setting. We also introduce some relevant results on random graphs that will be used later in the proofs. Then, in \cref{sec:Gnp}, we prove our main results on $G(n,p)$. We follow in \cref{sec:mindeg} by investigating analogues of our questions in the setting of graphs of high minimum degree and discussing how this setting relates to the sparse random setting. We conclude by discussing some potential future directions.
\bigskip

\section{Preliminaries} \label{sec:prelims}
All graphs studied in this paper are finite and simple.
For a graph $G$, let $E(G)$ and $V(G)$ denote the sets of edges and vertices of $G$, respectively, and write $e(G)=|E(G)|$ and $v(G)=|V(G)|$. 
Given subsets of vertices $S,T\subseteq V(G)$, let $G[S]$ denote the induced subgraph of $G$ on the vertex set $S$ (i.e. $G[S]$ is the graph on $S$ such that two vertices $v,w\in S$ are adjacent in $G[S]$ if and only if they are adjacent in $G$). Let $E_G(S)=E(G[S])$, $e_G(S)=|E_G(S)|$, and let $e_G(S,T)$ denote the number of ordered pairs $(s,t)\in S\times T$ such that $s,t$ are adjacent in $S$. We write $e(S,T)=e_G(S,T)$, $E(S)=E_G(S)$, and $e(S)=e_G(S)$ when the choice of $G$ is clear from context. Given two graphs $G$ and $M$ on the same set of vertices $V$, we write $G\cap M$ for the graph on vertex set $V$ whose edge set is $E(G)\cap E(M)$.

For real-valued functions $f$ and $g$, we write $f=o(g)$ or $f\ll g$ if $\lim_{n\to \infty}\frac{f(n)}{g(n)} = 0$, and we write $f=\omega(g)$ or $f\gg g$ if $\lim_{n\to\infty}\frac{f(n)}{g(n)} = \infty$. Unless otherwise noted, all such asymptotic notation will be in terms of the variable $n$ going to infinity.

To extend results about edge colorings of complete graphs to the random graph setting, we will frequently apply the following form of the Chernoff bound:
\begin{prop}[Chernoff bound]
    \label{prop:chernoff}
    Let $X=X_1+\cdots+X_n$, where the $X_i$ are i.i.d. Bernoulli random variables. Let $\mu=\EE[X]$. Then for any $\delta > 0$,
    \[
        \qquad \PP(X \ge (1+\delta) \mu)\le e^{-\delta^2 \mu/(2+\delta)},
    \]
    while for $\delta \in (0, 1)$,
    \[
        \PP(X \le (1-\delta) \mu)\le e^{-\delta^2 \mu/2}.
    \]
    In particular, for $\delta \in (0,1)$,
    \[\PP(|X-\mu|\ge \delta \mu)\le 2e^{-\delta^2 \mu/3}.\]
\end{prop}

\subsection{Upper bounds and Gy\'arf\'as's construction}
Gy\'arf\'as's construction \cite{gya} gives, for certain values of $r$, an upper bound corresponding to the conjectured lower bound in \cref{conj:rKn}, showing that the conjecture, if true, is tight for infinitely many values of $r$. We summarize the relevant details of the construction below.

The construction depends on the existence of a certain family of objects called affine planes. An affine plane of order $t$ consists of a set of $t^2$ points and $t^2+t$ lines, where every pair of points shares exactly one line, and the set of lines can be partitioned into $t+1$ classes (or ``directions'') of parallel lines, such that each pair of lines intersects in one point if they are not parallel (and no points if they are parallel). Affine planes of order $t$ are known to exist whenever $t$ is a prime power; it is open whether affine planes of any other order exist. 

When an affine plane of order $r-1$ exists and $n=(r-1)^2$, we can construct an $r$-coloring of $K_n$ as follows: 
The vertices are the points of an affine plane of order $r-1$, 
and the edge between any pair of vertices is colored based on the direction of the line 
between the corresponding points. Examples for $r-1=2$ and $r-1=3$ are shown below.
\begin{center}
    \includegraphics{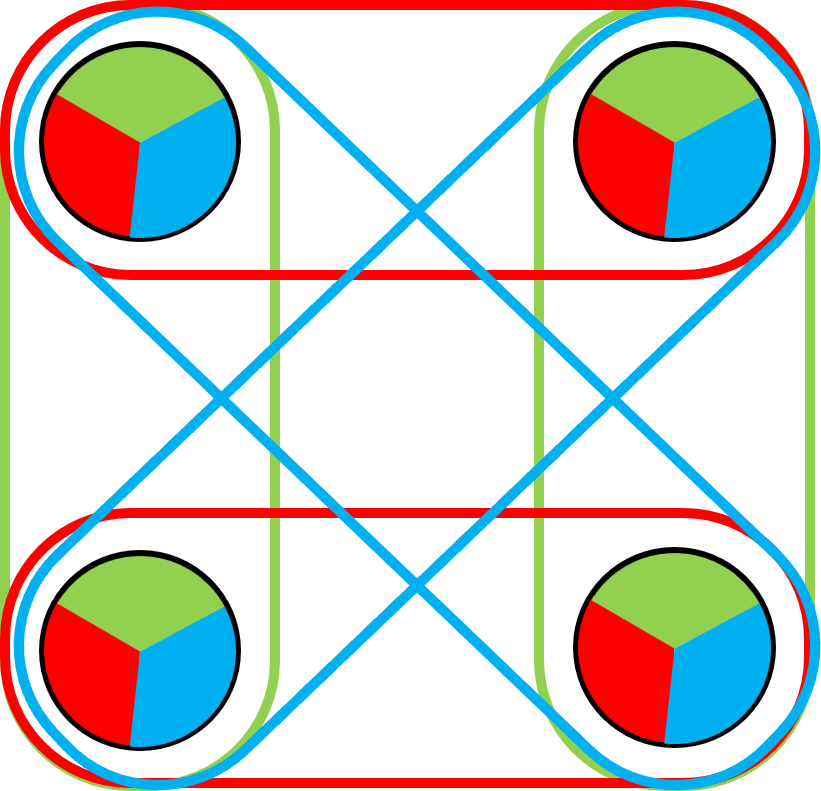}
    \includegraphics{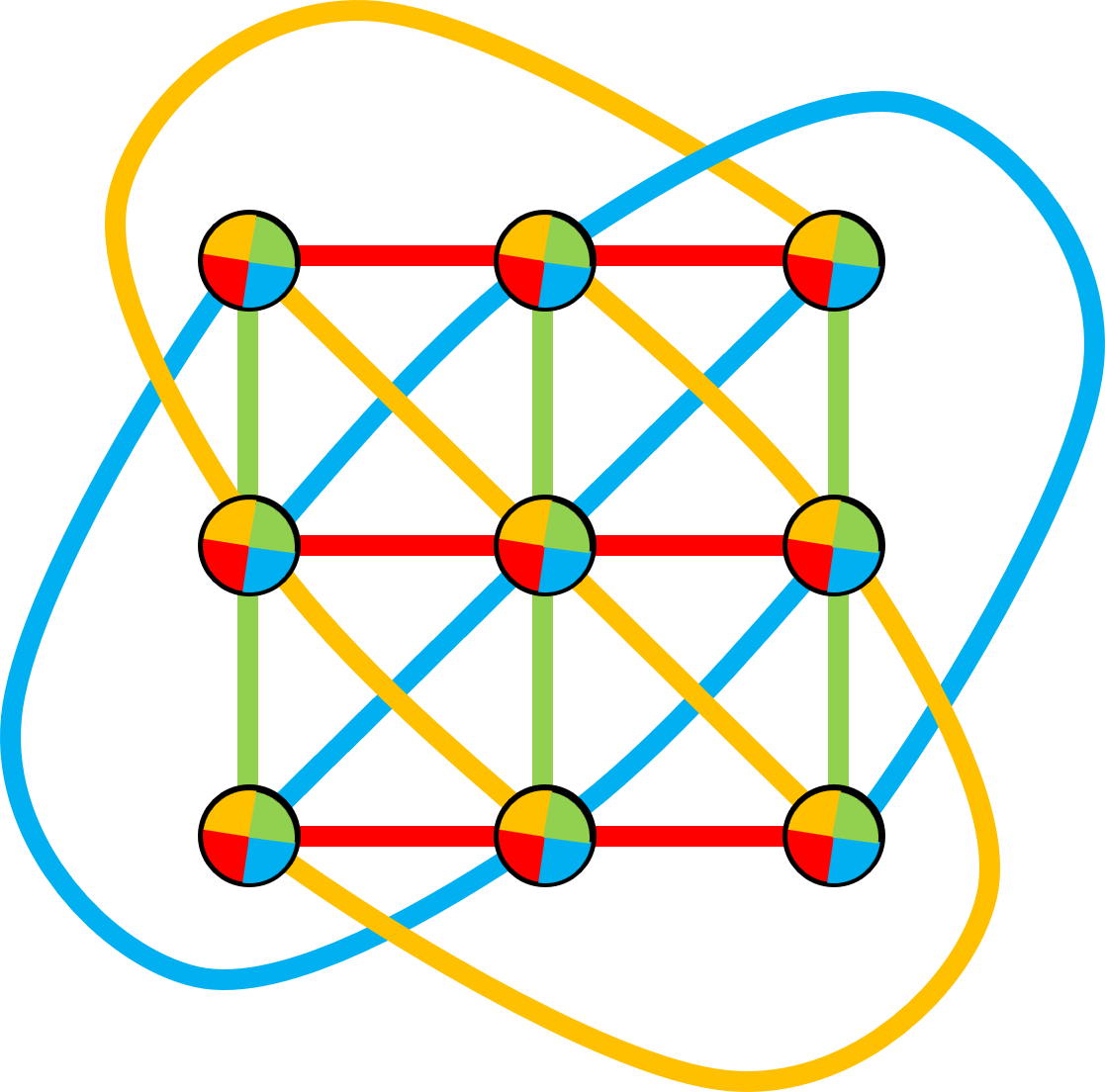}
\end{center}

To extend this coloring to larger values of $n$, we partition the vertices of $K_n$ into $(r-1)^2$ clusters, each corresponding to a point in the affine plane of order $r-1$.
For each edge between two distinct clusters, we color it based on the direction of the line between the two corresponding points in the affine plane.
Within each cluster, we color the edges such that each cluster has a roughly equal number of edges of each color.
Each connected component in a given color $k$ then consists of the clusters corresponding to the points on one line in the affine plane with the direction corresponding to $k$. The number of components then equals the number of lines in the affine plane of order $r-1$, which is $r(r-1)$. By symmetry, each component contains roughly the same number of edges, so each component contains roughly $\frac{1}{r^2-r}\binom{n}{2}$ edges.

We can use the same construction to show an analogous upper bound in the random graph setting, by identifying $G=G(n,p)$ as a subgraph of $K_n$ and taking the coloring induced by the $r$-edge-coloring of $K_n$ given by Gy\'arf\'as's construction. For fixed $r$ and large $n$, if $p$ is not too small (e.g. $p=\omega(1/n)$ suffices), 
each monochromatic connected component in the $r$-edge-coloring of $K_n$
will have, with high probability, close to a $p$ fraction of its edges remain in $G$, meaning that for any constant $\varepsilon>0$, with high probability, $G=G(n,p)$ will have the property that for some $r$-coloring of its edges, every monochromatic connected component has at most $(1+\varepsilon)p\frac{1}{r^2-r}\binom{n}{2}$ edges.

\subsection{Lower bounds in the complete graph setting}

An important tool in the proof of \cref{thm:rKn} in \cite{sluo} is the following inequality, which gives a density version of the relevant coloring results about sizes of connected components in subgraphs. Recall that a \emph{complete $k$-partite graph} is a graph on a vertex set $V=V_1\cup\cdots\cup V_k$, where the $V_i$ are pairwise disjoint, and a pair of vertices $v\in V_i$, $v'\in V_j$ is connected by an edge if and only if $i\neq j$. A graph $M$ is called a \emph{complete multipartite graph} if it is a complete $k$-partite graph for some $k\ge 2$.
\begin{theorem}[\cite{sluo}]
    \label{thm:rpart-comp}
    In a subgraph $H$ of a complete multipartite graph $M$, there must be a connected component 
    with at least $\frac{e(H)^2}{e(M)}$ edges.
\end{theorem}

This result, in turn, follows from the following even simpler looking (yet still nontrivial) inequality, which is Corollary~6 in \cite{sluo}.

\begin{prop}[\cite{sluo}]\label{prop:rpart-cor}
    Let $M$ be a complete multipartite graph. For any $S,T\subseteq V(M)$, we have
    \[
        e(S,T)^2\ge 4e_M(S) e_M(T).
    \]
\end{prop}

In order to prove \cref{thm:main}, we will need to adapt the proofs of \cref{prop:rpart-cor} and \cref{thm:rpart-comp}, in order to extend them to the random graph setting.

\section{Proof of results on random graphs}
\label{sec:Gnp}

Our first aim is to prove an analogue of \cref{prop:rpart-cor} for random graphs.
 We begin with a series of technical results that are essentially direct applications of the Chernoff bound.

\begin{prop}
    \label{prop:density-control}
    For any choice of $p=p_n$ such that $p=\omega(\frac{\log n}{n})$ and any $\varepsilon=\varepsilon_1\in (0,1)$, the following holds as $n$ goes to infinity. 
     In $G=G(n, p)$, with high probability, for every graph $H$ on $V(G)$ satisfying $e(H)\ge \varepsilon \binom{n}{2}$ that is either a clique on a subset of vertices or a complete multipartite graph on $V(G)$,
    \[(1-\varepsilon)\mu \le e(H\cap G) \le (1+\varepsilon)\mu,\]
    where $\mu=\EE[e(H\cap G)]=pe(H)$.
\end{prop}
\begin{proof}
    For any $H\subseteq E(K_n)$ and $\delta\in (0,1)$, the Chernoff bound yields
    \[
    \PP(|e(H\cap G)-\mu|\ge \delta \mu) \le e^{-\delta^2 \mu/3}.
    \]
    When $|H|\ge \varepsilon \binom{n}{2}$ and $\delta=\varepsilon$, this yields
    \[
        \PP(|e(H\cap G)-\mu|\ge \varepsilon \mu) \le \exp(-\varepsilon^2 p \varepsilon \binom{n}{2}/3)\le \exp(-(1+o(1))\varepsilon^3 p n^2/6).
    \]
    Since there are at most $2^n+n^n\le n^{2n}$ possible choices of $H$, the overall probability that there exists some valid choice of $H$ with $|e(H)-\mu|\ge \varepsilon \mu$ is then bounded above by
    \[
        \exp(-(1+o(1))\varepsilon^3 p n^2/6+2n\log n).
    \]
    For this probability to go to $0$, it suffices to have
    \[
        \varepsilon^3 p n^2 / (12n\log n) > 1.
    \]
    Since $p=\omega(\frac{\log n}{n})$, for large enough $n$ we have $\varepsilon > 3(\frac{\log n}{pn})^{1/3}$, which gives the desired inequality.
\end{proof}

\begin{prop}
    \label{prop:degree-bound}
    For any $c_0\in (0, 1)$, any choice of $p=p_n$ such that $p=\omega(\frac{\log n}n)$, and any $\varepsilon=\varepsilon_2\in (0,1)$, the following holds as $n$ goes to infinity. 
    In $G := G(n, p)$, with probability $1-o(\frac{1}{n})$, every subgraph $H$ with average degree $\bar d(H)\ge c_0 p(n-1)$ has $v(H)> \bar d(H)(1-\varepsilon)\frac{n}{p(n-1)}$.
\end{prop}
\begin{proof}
    Condition on the high probability event in \cref{prop:density-control} with parameters $p=p_n$ and $\varepsilon_1=\min(\varepsilon,(1-\varepsilon)^2c_0^2)$ to be determined later. 
    It suffices to consider the induced subgraphs $H:=G[S]$ on vertex sets $S\subseteq V(G)$. Fix some such $S$, and let $s:=|S|$. If $s\ge (1-\varepsilon)c_0 n$, then $e_{K_n}(S)=\binom{s}{2}\ge (1-\varepsilon)^2 c_0^2\binom{n}{2}\ge \varepsilon_1 \binom{n}{2}$. As long as $n$ is large enough, \cref{prop:density-control} applied to the clique on $S$ yields that $e(H)=e_G(S)\le (1+\varepsilon)p\binom{s}{2}$, which means $\bar d(H)=\frac{2e(H)}{v(H)}\le (1+\varepsilon)p(s-1)$. Then $s\ge \frac{\bar d(H)}{(1+\varepsilon)p}\frac{s}{s-1}\ge (1-\varepsilon)\bar d(H) \frac{n}{p(n-1)}$, as desired.
    
    Thus it suffices to show that with high probability, every $S$ with $\bar d(G[S])\ge c_0 p(n-1)$ has $s:=|S|\ge (1-\varepsilon)c_0 n$. Fix some $S$ with $s<(1-\varepsilon)c_0 n$. Note that a subgraph with average degree at least $c_0 p(n-1)$ must have at least $c_0 p(n-1)+1$ vertices, so we can assume $s\ge c_0 p(n-1)+1\ge c_0 pn$. 
    By the Chernoff bound, we obtain that 
    \begin{align*}
        \PP(\bar d(H)\ge c_0 p(n-1))&\le \PP(e(H)\ge \frac{1}{2}sc_0 p(n-1)) \\
        &= \PP (e(H)\ge \frac{c_0 (n-1)}{s-1}\EE[e(H)]) \\
        &\le e^{-\frac{\left(\frac{c_0 (n-1)}{s-1}-1\right)^2}{\frac{c_0 (n-1)}{s-1}+1}p\binom{s}{2}} \le e^{-\frac{\varepsilon}{2-\varepsilon}\frac{\varepsilon c_0 n}{s-1}p\binom{s}{2}} \le e^{-\frac{\varepsilon^2}{4}(n-1)c_0sp}.
    \end{align*}
    For any fixed value of $s$, a union bound over all $S$ of size $s$ shows that the probability that $G$ has a subgraph on $s$ vertices with average degree more than $c_0 p(n-1)$ is at most
    \[{n\choose s}e^{-\frac{\varepsilon^2}{4}(n-1)c_0sp} \le e^{s(1+\log n - \log s)-\frac{\varepsilon^2}{4}(n-1)c_0sp} \le e^{s(1+\log n-\frac{\varepsilon^2}{4}(n-1)c_0p)}.\]
    Since $p=\omega(\frac{\log n}{n})$, for $n$ large enough such that $p>\frac{8}{\varepsilon^2 c_0}\frac{\log n}{n-1}$ and $\log n >1$, this last quantity is maximized when $s$ is as small as possible, at a value of $e^{c_0pn(1+\log n-\frac{\varepsilon^2}{4}(n-1)c_0p)}$.
    Taking a union bound over all possible values of $s<(1-\varepsilon)c_0 n$, the probability of any such subgraph existing is at most $ne^{c_0pn(1+\log n-\frac{\varepsilon^2}{4}(n-1)c_0p)} = e^{\log n + c_0pn(1+\log n-\frac{\varepsilon^2}{4}(n-1)c_0p)}$. Since $p=\omega(\frac{\log n}{n})$, for large enough $n$ this upper bound is $e^{-\omega((\log n)^2)}= o(\frac{1}{n})$, as claimed.
\end{proof}

We are now ready to prove the following random graph analogue of \cref{prop:rpart-cor}.
\begin{theorem}
	\label{thm:corollary}
    For any choice of $p=p_n$ such that $p=\omega(\frac{\log n}n)$, any constant $c \in (0, 1)$, and any $\varepsilon = \varepsilon_3\in (0,1)$, the following holds as $n$ goes to infinity. In $G:=G(n,p)$, with probability $1-o(\frac1n)$, the following holds for every complete multipartite graph $M$ on the vertices of $G$:
 
    For every pair of subsets $S, T\subseteq V(G)$ with $e_M(S, T) \ge cn^2$, in the graph $G'=G\cap M$, we have 
    \[
        e_{G'}(S,T)\ge (1-\varepsilon)pe_M(S,T),
    \]
    and furthermore,
    \[e_{G'}(S, T)^2\ge 4(1-\varepsilon)e_{G'}(S) e_{G'}(T).\]
\end{theorem}
The first inequality is another simple Chernoff bound; we will prove the second inequality by arguing that, without the factor of $1-\varepsilon$, the inequality is true 
in expectation, and that with high probability, the actual values of both sides of the equation are close to their expected values.

\begin{proof}
    There exists a constant $d$ such that the number of ways to choose $M$, $S$, and $T$ is at most $e^{dn\log n}$. Fix such a choice of $M$, $S$, and $T$. 
    As in the proof of \cref{prop:density-control}, by the Chernoff bound, we have 
    \begin{align*}
        \mathbb{P}(e_{G'}(S, T)<(1-\varepsilon/4)\EE[e_{G'}(S, T)]) &\le e^{-\frac{\varepsilon^2\mathbb{E}[e_{G'}(S, T)]}{32}}\le e^{-\frac{\varepsilon^2 p c n^2}{32}}.
    \end{align*}
    Since $p=\omega(\frac{\log n}{n})$, taking a union bound over all choices of $M$, $S$, and $T$ gives an upper bound of
    \[
        e^{dn\log n} e^{-\frac{\varepsilon^2 p c n^2}{32}}=e^{-\omega(n\log n)}=o\left(\frac{1}{n}\right),
    \]
    so that with $1-o(\frac1n)$ probability,
    \[e_{G'}(S, T) \ge (1-\varepsilon/4) \EE[e_{G'}(S, T)]= (1-\varepsilon/4)pe_M(S,T),\]
    for all choices of $M$, $S$, and $T$ such that $e_M(S,T)\ge cn^2$, which proves the first inequality. By \cref{prop:rpart-cor}, $e_M(S, T)^2\ge 4e_M(S)e_M(T)$, so we also have that with $1-o(\frac1n)$ probability,
    \begin{align*}
        e_{G'}(S, T)^2 & \ge (1-\varepsilon/4)^2 p^2e_M(S, T)^2\ge (1-\varepsilon/4)^2 p^2 4e_M(S)e_M(T)\\
        &= 4(1-\varepsilon/4)^2 \mathbb{E}[e_{G'}(S)]\mathbb{E}[e_{G'}(T)],
    \end{align*}
    for all such $M$, $S$, and $T$. Note that $e_M(S,T)\ge cn^2$ implies $|S|,|T|\ge cn$, so that
    \[\mathbb{P}(e_{G'}(S) > (1+\varepsilon/4)\mathbb{E}[e_{G'}(S)]) \le 
    e^{-\frac{\varepsilon^2\mathbb{E}[e_{G'}(S)]}{48}}\le e^{-\frac{\varepsilon^2pc^2 n^2}{4}}=e^{-\omega(n\log n)},\]
    and similarly for $e_{G'}(T)$. 
    Therefore, again with $1-o(\frac1n)$ probability, we have
    \[e_{G'}(S, T)^2\ge \frac{(1-\varepsilon/4)^2}{(1+\varepsilon/4)^2}4e_{G'}(S)e_{G'}(T)\ge 4(1-\varepsilon)e_{G'}(S)e_{G'}(T),\]
    for all such choices of $M$, $S$, and $T$, as desired.
\end{proof}

Next, we prove the following random graph analogue of \cref{thm:rpart-comp}.
\begin{theorem}
	\label{thm:multipartite}
    For any choice of $p=p_n$ such that $p=\omega(\frac{\log n}n)$, any constants $c_1,c_2 \in (0, 1)$, and any $\varepsilon\in (0,1)$, the following holds as $n$ goes to infinity. With probability $1-o(\frac1n)$, in $G:=G(n, p)$, the following holds for every complete multipartite graph $M$ on $V(G)$ with minimum degree at least $c_1 n$:
    
    Let $G':=G\cap M$. Then every subgraph $H$ of $G'$ with $\frac{e(H)}{e(G')}\ge c_2$ has a connected component $H'$ with $e(H')\ge \frac{e(H)^2}{e(G')}(1-\varepsilon)$.
\end{theorem}
\begin{proof}
    Condition on the high probability events of \cref{prop:density-control}, \cref{prop:degree-bound}, and \cref{thm:corollary}, with parameters $c_0=\varepsilon(1-\varepsilon)^2\frac{c_1 c_2}{3}$, $c=\varepsilon(1-\varepsilon)^2\frac{c_1^2 c_2}{3}$, $\varepsilon_1=\varepsilon_2=\varepsilon_3=\varepsilon/3$. Let $H_1,\dots, H_k$ be the connected components of $H$. For any set $S\subset V:=V(G)$, let $f(S)=\frac12e_{G'}(S, V)$. In particular, we have $f(H_1)+\cdots+f(H_k)=f(V)=\frac{1}{2}e_{G'}(V,V)=e(G')$. 
    Then, by the generalized mediant inequality, for some $H'=H_\ell$, we have
    \[\frac{e(H')}{f(V(H'))} \ge \frac{e(H_1)+\cdots+e(H_k)}{f(V(H_1))+\cdots+f(V(H_k))} = \frac{e(H)}{e(G')}\ge c_2.\] 
    If $v(H')\ge \frac{c}{c_1}n$, then since $M$ has minimum degree $\delta(M)\ge c_1 n$, we obtain $e_M(V(H'),V)\ge cn^2$. \cref{thm:corollary} then yields
    \[
        e_{G'}(V(H'),V)^2\ge 4(1-\varepsilon_3)e_{G'}(V(H'))e_{G'}(V)=4(1-\varepsilon_3)e(H')e(G'),
    \]
    so that
    \[
        e(H')\ge \frac{e(H')}{f(V(H'))} f(V(H')) \ge \frac{e(H)}{e(G')}\sqrt{e(H')e(G')(1-\varepsilon_3)},
    \]
    which rearranges to give
	\[e(H')\ge \frac{e(H)^2}{e(G')}(1-\varepsilon_3)\ge \frac{e(H)^2}{e(G')}(1-\varepsilon),\]
    in the high probability event we have conditioned on, as desired. So it suffices to reduce to the case where $v(H')\ge \frac{c}{c_1}n$ for some $H'$ such that $\frac{e(H')}{f(V(H'))}$ is large. 
    
    Consider the vertex set
    \[
        T=\bigcup_{i:|V(H_i)|<\frac{c}{c_1}n}V(H_i).
    \]
    First suppose that $|T|\ge \frac{c}{c_1}n$. In this case, similarly to before, we obtain $e_M(T,V)\ge cn^2$, so that $e_{G'}(T,V))\ge (1-\varepsilon_3)pe_M(T,V)\ge (1-\varepsilon_3)p|T|c_1 n$. Then the average degree of the induced subgraph $H[T]$ is
    \[
        \bar d_{H}(T)=\frac{2e_{H}(T)}{|T|}\ge \frac{e_{H}(T)}{\frac{1}{2}e_{G'}(T,V)}(1-\varepsilon_3)pc_1 n.
    \]
    If $\frac{e_H(T)}{f(T)}\ge \frac{e(H)}{e(G')}\ge c_2$, then
    \[
        \bar d_{H}(T)\ge (1-\varepsilon_3)pc_1 c_2 n \ge c_0 pn,
    \]
    so that $H[T]$ contains a connected component $H_i$ with $\bar d(H_i)\ge c_0 pn$. By \cref{prop:degree-bound}, this implies $v(H_i)\ge \bar d(H_i)(1-\varepsilon_3)\frac{n}{p(n-1)}\ge (1-\varepsilon_3)^2 c_1 c_2 n>\frac{c}{c_1}n$,
    contradicting the definition of $T$. So we must have $\frac{e_H(T)}{f(T)}< \frac{e(H)}{e(G')}$. But then there is some connected component $H'$ with $V(H')\subseteq V\setminus T$ such that
    \[
        \frac{e(H')}{f(V(H'))} \ge \frac{e_{H}(V\setminus T)}{f(V\setminus T)} \ge \frac{e(H)}{e(G')},
    \]
    with $v(H')\ge \frac{c}{c_1}n$, so that we can conclude as before that $e(H')\ge \frac{e(H)^2}{e(G')}(1-\varepsilon)$, as desired.

    It remains to consider the case where $|T|<\frac{c}{c_1}n$. Since $\frac{c}{c_1}\ge (1-\varepsilon_2)c_0$, we must have $\bar d_{H}(T)<(1-\varepsilon_2)^{-1} \frac{c}{c_1}p(n-1)$, or else \cref{prop:degree-bound} leads to a contradiction. So $e(T)<\frac{c^2}{2(1-\varepsilon_2)c_1^2}pn(n-1)$. Then
    \begin{align*}
        \frac{e_{H}(V\setminus T)}{f(V\setminus T)}&\ge \frac{e(H)-e_{H}(T)}{f(V)}\ge \frac{e(H)}{e(G')}-\frac{\frac{c^2}{(1-\varepsilon_2)c_1^2}p\binom{n}{2}}{(1-\varepsilon_1)p\binom{n}{2}}\\
        &\ge \left(1-\frac{c^2}{(1-\varepsilon)^2c_1^2c_2}\right)\frac{e(H)}{e(G')} \ge (1-\varepsilon/3)\frac{e(H)}{e(G')}.
    \end{align*}
    We can then argue as before that for some connected component $H'$ of $H$ with $v(H')\ge \frac{c}{c_1}n$, we have $\frac{e(H')}{f(V(H'))}\ge (1-\varepsilon/3)\frac{e(H)}{e(G')}\ge (1-\varepsilon/3)c_2$. obtaining as before that $e_{G'}(V(H'),V)^2\ge 4(1-\varepsilon_3)e(H')e(G')$, we conclude that
    \[
        e(H')=\frac{e(H')}{f(V(H'))} f(V(H')) \ge(1-\varepsilon/3) \frac{e(H)}{e(G')}\sqrt{e(H')e(G')(1-\varepsilon_3)},
    \]
    which rearranges to give
    \[
        e(H')\ge (1-\varepsilon/3)^2(1-\varepsilon_3)\frac{e(H)^2}{e(G')}\ge (1-\varepsilon)\frac{e(H)^2}{e(G')},
    \]
    as desired. 
    Thus in every case, in the high probability event we have conditioned on, we obtain the desired bound.
\end{proof}

We finally have all the tools we need to prove our main result, \cref{thm:main}.

\begin{proof}[Proof of \cref{thm:main}]
    Let $z_0=\frac{1}{r^2-r+\frac{5}{4}}$, $c_1=\frac{1}{r+1}$, $c_2=\frac{1}{r-1}$, and $c_0=\frac{1-\varepsilon_0}{r}$, where $\varepsilon_0\le c_1(1-c_1)$ is some constant depending on $r$ to be chosen later. 
    For a given choice of $p=p_n=\omega(\frac{\log n}{n})$, condition on the high probability events in the conclusions of \cref{prop:density-control}, \cref{prop:degree-bound}, and \cref{thm:multipartite} with parameters $c_0,c_1,c_2$ as chosen above, and $\varepsilon_1=\varepsilon_2=\varepsilon=\varepsilon_0/3$. 
    In particular, we can assume that $(1-\varepsilon)p\binom{n}{2}\le  e(G)\le (1+\varepsilon)p \binom{n}{2}$.
    
    Fix an $r$-edge-coloring of $G$. Without loss of generality, let red be the color with the most edges. Let $R$ be the subgraph of $G$ consisting of all its red edges, so $e(R)\ge \frac{1}{r}e(G)\ge (1-\varepsilon)p\frac{1}{r} \binom{n}{2}$. Let $x=\frac{e(R)}{e(G)}\ge \frac{1}{r}$. Let the red connected components be $R_1,\dots,R_k$, with $v(R_1)\ge \cdots \ge v(R_k)$. Let $z$ be the fraction of the edges of $G$ that are in the monochromatic component with the most edges. It suffices to show that $z\ge z_0$, conditioned on the high probability events specified earlier. Note that for some $i\in [1,k]$ we have $\bar d(R_i)\ge \bar d(R)\ge \frac{1-\varepsilon}{r}p(n-1)\ge c_0 p(n-1)$, so that \cref{prop:degree-bound} yields $v(R_i)\ge (1-\varepsilon)^2\frac{n}{r}$, and thus $z\ge \frac{e(R_i)}{\binom{n}{2}}\ge \frac{(1-\varepsilon)^3}{r^2}\ge \frac{1-\varepsilon_0}{r^2}$.
    
    First, suppose that $v(R_1)\ge (1-\frac{1}{r+1})n$. Then by \cref{prop:density-control}, we have
    \[
        e(R_2)+\cdots+e(R_k)\le e_G(V(G)\setminus V(R_1))\le (1+\varepsilon)p \binom{\frac{1}{r+1} n}{2}\le (1+\varepsilon) p \frac{1}{(r+1)^2} \binom{n}{2},
    \]
    which means
    \[
    e(R_1)\ge (1-\varepsilon)p\frac{1}{r} \binom{n}{2} - (1+\varepsilon) p \frac{1}{(r+1)^2} \binom{n}{2} \ge \frac{1+\varepsilon}{r^2-r+\frac{5}{4}}p\binom{n}{2},
    \]
    for $r\ge 2$ and sufficiently small $\varepsilon$, so that $z\ge z_0$ as desired. 
    Thus, we can assume that $v(R_1)< (1-\frac{1}{r+1})n=(1-c_1)n$. 

    For $0\le j\le k-1$, let $G_j=G[V(R_{j+1})\cup\dots\cup V(R_k)]$ and $H_j=G_j\cap R$. Let $c=\sqrt{z\frac{1+\varepsilon_0}{1-\varepsilon_0}}$. Since $z\ge \frac{1-\varepsilon_0}{r^2}$, we have $c\ge c_0$. If any component $R_i$ has average degree at least $pc(n-1)\ge pc_0(n-1)$, then the assumption of \cref{prop:degree-bound} gives $v(R_j)> (1-\varepsilon)cn$, so that $e(R_j)> (1-\varepsilon_0)c^2 p\binom{n}{2}= z(1+\varepsilon_0)p\binom{n}{2}\ge ze(G)$, contradicting our choice of $z$. Thus we can assume every $R_i$, and hence every $H_j$, has average degree less than $pc(n-1)\le (1-\varepsilon_0)^{-1}p\sqrt{z}(n-1)$. In particular, $j=0$ gives
    \[
    x e(G)=e(R) < (1-\varepsilon_0)^{-1}p\sqrt{z} \binom{n}{2}\le \frac{\sqrt{z}}{(1-\varepsilon_0)^2}e(G),
    \]
    so that $z\ge (1-\varepsilon_0)^4x^2\ge (1-2\varepsilon_0)^2 r^{-2}$. Let $\delta=r-\frac{1-2\varepsilon_0}{\sqrt{z}}$, so $z=\frac{(1-2\varepsilon_0)^2}{(r-\delta)^2}$. If $\delta\ge 1$, then $z\ge z_0$ as long as we pick $\varepsilon_0$ small enough in terms of $r$, so we can assume $\delta<1$.
    More generally, since $e(R_i)\le ze(G)$ for $1\le i\le k$, we have for every $j$ that
    \[
        e(H_j)=\sum_{i=j+1}^k e(R_i)\ge (x-jz)e(G),
    \]
    meaning 
    \[
        (1-\varepsilon_0)^{-1}p\sqrt{z}(n-1) > \frac{2e(H_j)}{v(H_j)}\ge \frac{2(x-jz)e(G)}{n-\sum_{i=1}^j v(R_i)},
    \]
    which rearranges to give
    \begin{align*}
        \sum_{i=1}^j v(R_i)&< n-\frac{2(x-jz)e(G)}{(1-\varepsilon_0)^{-1}p\sqrt{z}(n-1)}\\
        &\le n (1-\frac{(1-\varepsilon_0)^2 (x-jz)}{\sqrt{z}})\le (1+2\varepsilon_0)(1-x/\sqrt{z}+j\sqrt{z})n.
    \end{align*}
	The same smoothing argument as in the proof of \cite[Lemma~3.2]{sluo} then yields
	\[\sum_{i=1}^j{{v(R_i)}\choose 2}\le \frac{n^2((1-x/\sqrt{z}+\sqrt{z})^2+(x/z-1)z)-n}2(1+5\varepsilon_0).\]
	Let $M$ be the complete $k$-partite graph on $V(G)$ with $V(R_1),\dots,V(R_k)$ as the independent sets.
	Because $G':=M\cap G$ has no red edges, one of the other colors (say, blue) has at least $\frac1{r-1}e(G')$ edges in $G'$. Let $H$ be the subgraph of $G'$ consisting of all its blue edges, so $\frac{e(H)}{e(G')}\ge \frac{1}{r-1}=c_2$. Since $v(R_1)<(1-c_1)n$, the minimum degree of $M$ is at least $c_1 n$. Then, by \cref{thm:multipartite}, there is a connected component $H'$ of $H$
	with $e(H')\ge (1-\varepsilon) \frac{e(H)^2}{e(G')} \ge (1-\varepsilon)\frac1{(r-1)^2}e(G')$. Since $e(M)\ge c_1 n (1-c_1) n \ge \varepsilon_0 \binom{n}{2}$, \cref{prop:density-control} yields $e(G')\ge (1-\varepsilon)pe(M)$, so
    \begin{align*}
        z&\ge (1-\varepsilon)\frac{1}{(r-1)^2}\frac{e(G')}{e(G)}\ge (1-\varepsilon)^2 \frac{1}{(r-1)^2}p \frac{e(M)}{e(G)}\\
        &\ge (1-\varepsilon)^2\frac{1}{(r-1)^2}(1-((1-x/\sqrt{z}+\sqrt{z})^2+(x/z-1)z)(1+5\varepsilon_0))\\
        &\ge (1-7\varepsilon_0)\frac{1}{(r-1)^2}(1-(1-x/\sqrt{z}+\sqrt{z})^2-(x/z-1)z).
    \end{align*}
    This rearranges to yield 
    \begin{align*}
        -\frac{(r-1)^2z}{1-7\varepsilon_0}&\le (1-x/\sqrt{z}+\sqrt{z})^2+(x/z-1)z-1\\
        &=\frac{1}{z}x^2-(1+2/\sqrt{z})x+2\sqrt{z}.
    \end{align*}
    For $\varepsilon_0$ small enough in terms of $r$, we have $\frac{z}{2}+\sqrt{z}\ge x\ge \frac{1}{r}$, so the right hand side is at most
    \[
        \frac{1}{zr^2}-(1+2/\sqrt{z})\frac{1}{r}+2\sqrt{z}.
    \]
    Recall that $z = \frac{(1-2\varepsilon_0)^2}{(r-\delta)^2}$ for some $\delta\in(0,1)$. Substituting this in yields
    \begin{align*}
        -\frac{(1-2\varepsilon_0)^2(r-1)^2}{(1-7\varepsilon_0)(r-\delta)^2} &\le \frac{(r-\delta)^2}{(1-2\varepsilon_0)^2r^2} - \frac{(1-2\varepsilon_0)+2(r-\delta)}{(1-2\varepsilon_0)r}+\frac{2(1-2\varepsilon_0)}{r-\delta},
    \end{align*}
    which can be rearranged to get 
    \begin{align*}
        0&\le \frac{(1-2\varepsilon_0)^4}{1-7\varepsilon_0}r^2(r-1)^2\\
        +&(r-\delta)^4+2(1-2\varepsilon_0)^3(r-\delta)r^2-(1-2\varepsilon_0)^2r(r-\delta)^2-2(1-2\varepsilon_0)r(r-\delta)^3 \\
        &\le (r-\delta^2)^2-r(r^2-\delta^2)(1-2\delta)+4\varepsilon_0 r(r-\delta)^2(r-\delta+1),
    \end{align*}
    and thus, taking $\varepsilon_0\le (32(r+1))^{-4}$,
    \[
        1-2\delta \le \frac{(r-\delta^2)^2}{r(r^2-\delta^2)}+\frac{1}{8r^3}< \frac{1}{r}+\frac{1}{8r^3},
    \]
    so that $\delta >\frac{r-1}{2r}-\frac{1}{16r^3}$, and thus
    \[
        z=\frac{(1-2\varepsilon_0)^2}{(r-\delta)^2}\ge \frac{1-32^{-3}r^{-4}}{r^2-r+\frac{1}{4}+1-\frac{1}{4r}+\frac{1}{4r^2}+2\frac{r}{16r^3}}.
    \]
    For $r\ge 2$, it can be verified that this expression is at least $\frac{1}{r^2-r+\frac{5}{4}}=z_0$, as desired.
\end{proof}
\section{High minimum degree argument}
We now turn to the closely related setting of graphs of high minimum degree. Unlike in the problem of finding monochromatic components with many vertices, for the edge version of the problem it suffices to assume a global condition of an edge density close to $1$, rather than a local condition about degrees. We obtain a result that approximates the bound in \cref{thm:rKn}.
\label{sec:mindeg}
\begin{theorem}
    Let $r\ge 4$ be an integer and let $\beta\in (0,\frac{1}{r-1})$. Then for any graph $G$ on $n$ vertices with at least $(1-\beta){n\choose 2}$ edges, every $r$-coloring of the edges of $G$ contains a monochromatic component with at least
    $\frac{(1-\beta)^2}{r^2-r+\frac{5}{4}}\binom{n}{2}$
    edges.
\end{theorem}
\begin{proof}
    We adapt the proof of \cref{thm:rKn} in~\cite{sluo}, needing to make only a few adjustments. 
    Instead of considering an edge coloring of a subgraph of $K_n$, we essentially consider an edge coloring of $K_n$ that leaves up to a $\beta$ fraction of edges uncolored. Fix such a partial $r$-edge-coloring, let $z$ be the fraction of the edges of $K_n$ that are in the monochromatic component with the most edges, and let red be the color with the most edges. Let $R$ be the subgraph of $K_n$ consisting of all the edges colored red, and let $x=\frac{e(R)}{e(K_n)}\ge \frac{1-\beta}{r}$. By \cref{thm:rpart-comp} applied trivially with $M=K_n$, we have $z\ge x^2\ge \frac{(1-\beta)^2}{r^2}$.

    Let $R_1,\dots, R_k$ be the red connected components, with $v(R_1)\ge\cdots\ge v(R_k)$. For $j\in [0, k-1]$, let $G_j$ be the complete graph on $\bigcup_{i={j+1}}^k V(R_i)$, with its (partial) coloring induced by the coloring on $K_n$. Let $H_j=G_j\cap R$. 
    Applying \cref{thm:rpart-comp} with $M=G_j$ and $H=H_j$ yields a red component $R_i$ with $ze(K_n)\ge e(R_i)\ge \frac{e(H_j)^2}{e(G_j)}$. We have
    \[e(H_j) = \sum_{i=j+1}^k e(R_i) \ge (x-jz)e(K_n),\]
    while
    $e(G_j) = {{n-\sum_{i=1}^j v(R_i)}\choose 2}.$
    This is the exact same setup as in the proof of \cref{thm:rKn} given in \cite{sluo}. 
    Letting $v_i=\frac{V(R_i)}{n}$ and solving the optimization via \cite[Lemma~3.2]{sluo} yields
    \[\sum_{i=1}^m v_i^2\le (1-x/\sqrt{z}+\sqrt{ z})^2+(x/z-1)z ,\]
    so that
    \[\sum_{i=1}^m {{v(R_i)}\choose 2} \le \frac{n^2((1-x/\sqrt{z}+\sqrt{ z})^2+(x/z-1)z)-n}2.\]
    Let $M$ be the complete multipartite graph obtained by removing all edges within each $V(R_i)$ from $K_n$, with its (partial) coloring induced by the coloring on $K_n$. Then no edges of $M$ are colored red, and at most $\beta e(K_n)$ edges of $M$ are uncolored, so by the pigeonhole principle, some color class $H$ in $M$ has $e(H)\ge \frac1{r-1}(e(M)-\beta e(K_n))\ge \frac{1}{r-1}e(K_n)(1-\beta - ((1-x/\sqrt{z}+\sqrt{ z})^2+(x/z-1)z))$. By \cref{thm:rpart-comp}, some component $H'$ of $H$ has $e(H')\ge \frac{e(H)^2}{e(M)}\ge \frac{1}{r-1}e(H)$. Therefore,
    \[z\ge \frac{e(H')}{e(K_n)}\ge \frac{1}{(r-1)^2}(1-\beta - ((1-x/\sqrt{z}+\sqrt{ z})^2+(x/z-1)z)).\]
    Applying the same sequence of rearrangements and simplifications as in the proof of \cref{thm:rKn} (which are very similar to those in the proof of \cref{thm:main}), with the distinction being that we only have $x\ge \frac{1-\beta}{r}$ in this case, then yields
    \[
        -(r-1)^2 z \le (1-\beta)^2 \frac{1}{zr^2}-(1+2/\sqrt{z})\frac{1-\beta}{r}+2\sqrt{z}+\beta.
    \]
    Letting $z=\frac{(1-\beta)^2}{(r-\delta)^2}$ and further expanding and simplifying the above then gives
    \begin{align*}
        0 &\le (r-\delta)^4 +(r-1)^2 r^2 (1-\beta)^2-(1-\beta)r(r-\delta)^2-2(r-\delta)^3r+2(1-\beta)(r-\delta)r^2+\beta(r-\delta)^2 r^2\\
        &= (r-\delta^2)^2-r(r^2-\delta^2)(1-2\delta)+\beta((r-\delta)^2r^2+r(r-\delta)^2-2(r-\delta)r^2-(2-\beta)(r-1)^2 r^2) \\
        &\le (r-\delta^2)^2-r(r^2-\delta^2)(1-2\delta)+\beta r^2((r-\delta)^2-2(r-1)^2),
    \end{align*}
    where the last step uses the fact that $\beta\le \frac{1}{r-1}$. Since $r\ge 4$, the last term $\beta r^2((r-\delta)^2-2(r-1)^2)$ is negative, so we obtain $\delta > \frac{r-1}{2r}$ just like in the (fully colored) complete graph setting. Then we have $z\ge \frac{(1-\beta)^2}{(r-\frac{1}{2}+\frac{1}{2r})^2}\ge \frac{(1-\beta)^2}{r^2-r+\frac{5}{4}}$, as desired.
    
\end{proof}

The proof of \cite[Theorem~1.2]{sluo} likewise translates to the following result about the $3$-color case in the high minimum degree setting.

\begin{theorem}
Let $G$ be a graph on $n$ vertices where every vertex has degree at least $(1-\beta)n-1$, where $\beta \le \frac{1}{25}$. 
Then in every $3$-coloring of the edges of $G$, there is a monochromatic connected component with at least $\frac16 e(G)$ edges.
\end{theorem}
\begin{proof}
    Call the three colors red, green, and blue. Assume for the sake of contradiction that every monochromatic component has less than $\frac{1}{6}e(G)$ edges.
    
    First suppose that one of the colors (say, red) consists of exactly one component. Then the number of red edges is at most $\frac16e(G)$, so there are at least $\frac56e(G)$ total green and blue edges. Without loss of generality,
    there are at least $\frac5{12}e(G)$ green edges. By \cref{thm:rpart-comp}, there is a green component $H'$ with
    \[e(H')\ge \frac{(\frac{5}{12}e(G))^2}{{n\choose 2}} \ge \frac{25}{144}(1-\beta) e(G) \ge \frac16e(G),\]
    a contradiction.
    So, we can assume there are multiple components of each color. 
    Let red be the color with the most edges, so there are at least $\frac13 e(G)$ red edges.
    Let $H'$ be the red component with the largest average degree, so
    \[v(H')\ge \bar d(H')+1\ge \frac{1}{3}\bar d(G)+1\ge \frac{1}{3}(1-\beta)n.\]
    If $v(H')>\frac{n}{2}$, then $e(H')=\frac{1}{2}v(H')\bar d(H')\ge \frac{1}{6}e(G)$, so we can assume that $v(H')\le \frac{n}{2}$. 
    Let $V_1=V(H')$ and $V_2 = V(G)\setminus V_1$, and let $G'$ be the bipartite subgraph of $G$ consisting of all its edges between $V_1$ and $V_2$.
    
    Fix a green component $C_1$ of $G'$. If $|C_1\cap V_1|\ge |V_1|-3\beta n-1$, then the edges of $G'$ between $V_1\cap C_1$ and $V_2\setminus C_1$ are all blue.
    So, all the edges between $V_1\cap C_1$ and $V_2$ are split into two monochromatic components.
    One of those components has a number of edges that is at least 
    \[
    \frac 12 e_G(V_1\cap C_1, V_2) \ge \frac{1}{2}|V_1\cap C_1||V_2|-\beta (n-1)|V_1|\ge \frac{2}{9}(1-\beta)\binom{n}{2}-\beta \binom{n}{2}\ge  \frac 16 e(G),
    \]
    a contradiction. So, no green (or blue) component can cover all but at most $3\beta n$ vertices of $V_1$ (or, by symmetry, of $V_2$).
    
    Let $A_1=C_1\cap V_1$, $A_2=C_1\cap V_2$, $B_1 = V_1\setminus C_1$, and $B_2 = V_2\setminus C_1$.
    By assumption, $B_1$ and $B_2$ each have at least $3\beta n+1$ vertices.
    Since every vertex in $A_1$ is connected in blue to all but at most $\beta n$ vertices of $B_2$, by pigeonhole every pair of vertices in $A_1$ shares a neighbor in $B_2$.
    Thus, $A_1$ is contained in a single blue component. The same is true of $A_2$ by symmetry.
    Let $M_1$ (resp. $M_2$) be the subset of $B_1$ (resp. $B_2$) contained in the same blue component as $A_2$ (resp. $A_1$). Then $M_1$ and $M_2$ each have at least $2\beta n+1$ vertices, so the bipartite graph between $M_1$ and $M_2$ is connected.
    If any edge between $M_1$ and $M_2$ was blue, then there would be a blue component with more than $|V_1|-3\beta n-1$ vertices in $V_1$,
    a contradiction.
    So, $M_1\cup M_2$ is contained in a green component. Let $K$ be the subgraph of $G$ induced on $A_1 \cup A_2 \cup M_1 \cup M_2$, with its coloring induced from the coloring on $G$. Every edge between $A_1$ and $M_1$, or between $A_2$ and $M_2$, is red;
    every edge between $A_1$ and $A_2$, or between $M_1$ and $M_2$, is green; and every edge between $A_1$ and $M_2$, or between $A_2$ and $M_1$, is blue.

    Let $R=V(G)\setminus V(K)$, so $|R|\le 2\beta n$.
    Suppose for the sake of contradiction that there is another green component contained in $R$.
    Let $v$ be a vertex of $R$. Then, all of the edges between $v$ and $V(K)$ are red or blue.
    Without loss of generality, assume there is a red edge from $v$ to a vertex of $A_1$.
    Then, because $A_1$ is in a different red component from $A_2$ and $M_2$, the edges from $v$ to $A_2$ and $M_2$ are blue.
    However, $A_2$ and $M_2$ must be in different blue components, a contradiction.
    By similar logic, $R$ cannot contain any extra red or blue components.
    So, in $G$, there are exactly two components of each color, 
    and thus one of the $6$ components has at least $\frac16e(G)$ edges, as desired.
\end{proof}

\section{Acknowledgements}

This paper was written as part of the MIT PRIMES program. The authors would like to thank the organizers of the PRIMES program for their support and valuable feedback.

\end{document}